\documentclass[twoside, 11pt]{article}
\usepackage{amsmath,amssymb,amsthm,mathrsfs}
\usepackage[margin=2.5cm]{geometry} 
\usepackage{lipsum}
\usepackage{titlesec,hyperref}
\usepackage{fancyhdr}
\usepackage[numbers,sort&compress]{natbib}
\usepackage{color}

\fancypagestyle{plain}
{
\fancyhf{}

}

\titleformat{\subsection}{\it}{\thesubsection.\enspace}{1.5pt}{}
\titleformat{\subsubsection}{\it}{\thesubsubsection.\enspace}{1.5pt}{}

\newtheorem{theo}{Theorem}[section]
\newtheorem{lemm}[theo]{Lemma}
\newtheorem{defi}[theo]{Definition}

\newtheorem{rema}{Remark}[section]
\numberwithin{equation}{section}

\allowdisplaybreaks

\def\th2{\frac{\theta}{2}}

\begin{document}

\title{\hspace{-4mm} Global existence of weak solutions for generalized  quantum MHD equation}
\author{ Boling Guo$^1$, Binqiang Xie$^2*$}
\date{}
\maketitle
\begin{center}
\begin{minipage}{120mm}
\emph{\small $^1$Institute of Applied Physics and Computational Mathematics, China Academy of Engineering Physics,
 Beijing, 100088, P. R. China \\
$^2$Graduate School of China Academy of Engineering Physics, Beijing, 100088, P. R. China }
\end{minipage}
\end{center}

\footnotetext{Email: \it gbl@iapcm.ac.cn(B.L.Guo), \it xbq211@163.com(B.Q.Xie).}
\date{}

\maketitle

\begin{abstract}
We prove the existence of a weak solution to a generalized quantum MHD equation in a 2-dimensional periodic box for large initial data.  The existence of a global weak
solution is established through a three-level approximation, energy estimates, and weak convergence for
the adiabatic exponent $\gamma>1$.

\vspace*{5pt}

\vspace*{5pt}
\end{abstract}


\section{Introduction}
\quad The evolution of quantum MHD equations in $\Omega= T^{2}$ is described by the following system
\begin{subequations}\label{1.1}
\begin{eqnarray}
&\partial_{t}n +{\rm div}(n u) = 0, \label{1.1a} \\
&\partial_{t}(n u)+{\rm div}(n u\otimes u)
  +\nabla (P(n)+P_{c}(n))- 2 {\rm div} (\mu(n) D(u)) \\   \nonumber
 &- \nabla (\lambda (n) {\rm div} u)  - \frac{\hbar^{2}}{2} n \nabla (\varphi^{\prime}(n) \Delta \varphi(n))- (\nabla \times B)\times B=0, \label{1.1b} \\
&\partial_{t} B - \nabla \times (u\times B) + \nabla \times (\nu_{b}(\rho)B \times B)=0,\label{1.1c}\\
& n(x,0)= n_{0}(x), ~nu(x,0)=m_{0},\label{1.1d}\\
& B(x,0)= B_{0}(x),~{\rm div} B_{0}=0,\label{1.1e}
\end{eqnarray}
\end{subequations}
where the functions $n, u $ and $B$
represent the mass density,the velocity field, the magnetic field.
$P(n)=n^{\gamma}$ stands for the pressure, $P_{c}$ is a singular continuous function and called cold pressure. $\mu(n), \lambda(n)$ denote the fluid viscosity coefficient.  $\hbar>0$ quantum plank constant, $\nu_{b}$ the magnetic viscosity coefficient.

Our analysis is based on the following physically grounded assumptions:

[A1]The viscosity coefficient is determined by the Newton's rheological law
\begin{equation}\label{1.2}
\mu(n)=\mu_{0} n^{\alpha},~0<\alpha\leq 1,  \lambda(n)=2(n\mu^{\prime}(n)- \mu(n)).
\end{equation}
where $\mu$ and $\lambda$ are respectively the shear and bulk constant viscosity coefficients, the dispersion term $\varphi$ satisfying
\begin{equation}\label{1.3}
\varphi(n)=n^{\alpha},
\end{equation}

[A2]The cold pressure $P_{c}$ obeys the following growth assumption:
\begin{equation}\label{1.4}
\lim_{n\rightarrow 0}P_{c}(n)= + \infty,
\end{equation}
More precisely, we assume
\begin{equation}\label{1.5}
P_{c}^{\prime}(n)=\left\{
\begin{aligned}
&c_{1} n^{-\gamma^{-}-1}~~n\leq 1,\\
&c_{2} n^{\gamma-1} ~~n >1,
\end{aligned}
\right.
\end{equation}
where constants $\gamma^{-}, \gamma^{+}\geq 1, c_{1},c_{2}>0$.

[A3]The positive coefficient  $\nu_{b}$ is supposed to be a continuous function of the density, bounded from above and taking large values for small and large densities. More precisely, we assume that there exists $B>0$, positive constants $d_{0},d_{0}^{\prime}, d_{1},d_{1}^{\prime}$ large enough, $2\leq a<a^{\prime}<3$ and $b\in [0,\infty]$ such that
\begin{equation}\label{1.6}
\forall s<B,~~\frac{d_{0}}{s^{a}}\leq \nu_{b}(s)\leq \frac{d_{0}^{\prime}}{s^{a^{\prime}}}~~and ~~\forall s\geq B, d_{1}\leq \nu_{b}(s) \leq d_{1}^{\prime}s^{b}.
\end{equation}

Define function $H(n)$ and $\xi(n)$ as follow :
\begin{equation}\label{1.7}
\left\{
\begin{aligned}
&nH^{\prime}(n)-H(n)=P(n),~~nH^{\prime}_{c}(n)-H_{c}(n)=P_{c}(n)\\
&n\xi^{\prime}(n)=\mu^{\prime}(n),
\end{aligned}
\right.
\end{equation}

Now, we give the definition of a weak solution to \eqref{1.1}.

\begin{defi}
We call $(n, u, B)$ is as  a  weak solution to the problem \eqref{1.1}, if the following is satisfied.

(1)the density $n$ is a non-negative function satisfying the internal identity
\begin{equation}\label{1.8}
\int_{0}^{T} \int_{\Omega} n \partial_{t} \phi + n u\cdot \nabla \phi dx dt+ \int_{\Omega} n_{0} \phi(0)dx=0,
\end{equation}
for any test function $\phi\in C^{\infty}([0,T]\times \overline{\Omega}), \phi(T)=0$.

(2)  the momentum equation in \eqref{1.1b}
holds in $D^{\prime}((0,T)\times\Omega)$(in the sense of distributions), that means,
\begin{equation}\label{1.9}
\begin{aligned}
&\int_{\Omega}m_{0} \phi(0) dx+\int_{0}^{T} \int_{\Omega} n u \cdot \partial_{t} \phi + n (u\otimes u): \nabla \phi+ P {\rm div} \phi dx dt\\
&=\frac{\hbar^{2}}{2} \int_{0}^{T} \int_{\Omega} \varphi^{\prime}(n) \Delta \varphi(n) \nabla n \phi+ n \phi^{\prime} \Delta \varphi(n) {\rm div} \phi dx dt + 2\int_{0}^{T} \int_{\Omega} \mu(n) D(u) \nabla \phi dx dt\\
& + \int_{0}^{T} \int_{\Omega}\lambda(n) {\rm div} u {\rm div} \phi dx dt- \nu_{b} \int_{0}^{T} \int_{\Omega}(\nabla \times B) \times B \cdot \phi dx dt ,
\end{aligned}
\end{equation}
for any test function $\phi\in C^{\infty}([0,T]\times \overline{\Omega}), \phi(T)=0$.

(3) the magnetic field $B$ is a non-negative function satisfying
\begin{equation}\label{1.10}
\begin{aligned}
&\int_{\Omega} B_{0}  \phi(0)dx =\int_{0}^{T} \int_{\Omega} (B \cdot \partial_{t}\phi +  (u \times B) \cdot(\nabla \times \phi)- \nu_{b} \nabla B :\nabla \phi) dx dt ,
\end{aligned}
\end{equation}
for any test function $\phi\in C^{\infty}([0,T]\times \overline{\Omega}), \phi(T)=0$.

\end{defi}

\begin{rema}
If $\mu(n)=0, \lambda(n)=0, \alpha=\frac{1}{2}, B=0, P_{c}(n)=0, $, then quantum hydrodynamic equation \eqref{1.1} becomes
\begin{subequations}\label{1.11}
\begin{eqnarray}
&\partial_{t}n +{\rm div}(n u) = 0, \label{1.9a} \\
&\partial_{t}(n u)+{\rm div}(n u\otimes u)
  +\nabla P(n)- \frac{\hbar^{2}}{2} n \nabla (\frac{\Delta \sqrt{n}}{\sqrt{n}})=0, \label{1.9b}
\end{eqnarray}
\end{subequations}
If $\mu(n)=0, \lambda(n)=0, \alpha=1, B=0, P_{c}(n)=0, \nu=0$, then quantum hydrodynamic equation \eqref{1.1} becomes
\begin{subequations}\label{1.12}
\begin{eqnarray}
&\partial_{t}n +{\rm div}(n u) = 0, \label{1.10a} \\
&\partial_{t}(n u)+{\rm div}(n u\otimes u)
  +\nabla P(n)- \sigma n \nabla \Delta n=0, \label{1.10b}
\end{eqnarray}
\end{subequations}
\end{rema}

Now, we are ready to formulate the main result of this paper.

\begin{theo} (global existence for the quantum Euler model)
Let $\Omega =T^{2}$ be a periodic box.  Assume $ T>0$.
Let the initial data satisfy
\begin{equation}\label{1.13}
\left\{
\begin{aligned}
&\frac{1}{2}\int_{T^{2}} (\frac{|m|^{2}}{2n_{0}} dx + [H(n_{0})+ H_{c}(n_{0})]+ \frac{\hbar^{2}}{2} |\nabla \varphi(n_{0})|^{2} +
\int_{T^{d}} |B_{0}|^{2})dx <+\infty,\\
&\frac{\nabla \mu(n_{0})}{\sqrt{n_{0}}} \in L^{2}(\Omega), \
\end{aligned}
\right.
\end{equation}
 Then problem \eqref{1.1}-\eqref{1.10} posses at least one  global weak solution $n,u, B$.
\end{theo}

This paper is organized as follows.
In section $2$, we establish the global existence
of solutions to the Faedo-Galerkin approximation to \eqref{1.1}. In section $3$ we deduce the B-D entropy energy estimates, which is a key part in the analysis process. In section $4$ and $5$, we use the uniform estimates
 to recover the original system by vanishing the artificial viscosity and artificial pressure respectively,
  therefore the main theorem is proved by using the weak convergence method.

\section{Faedo-Galerkin approximation}
\quad In this section, we prove the existence of solutions to approximate viscous quantum Euler equations. We proceed similarly as [16. Chap. 7].
\subsection{Local existence of solutions.}
Let $T>0$, and let $(e_{k})$ be an orthonormal basis of $L^{2}(T^{d})$ which is also an orthogonal basis of $H^{1}(T^{d})$. Introduce the finite-dimensional space $X_{N}=span \{ e_{1}, e_{2}, ..., e_{N} \}$,  $N\in \mathbb{N}$. Let $(n_{0},u_{0}, B_{0})\in C^{\infty}(T^{d})^{3}$  be some initial data satisfying $n_{0}\geq \delta>0$ for $x\in T^{d}$ for some $\delta>0$, and let the velocity $u \in C^{0}([0,T]; X_{n})$ be given. We notice that $u$ can be written as
\begin{equation}\label{2.1}
u(x,t)= \sum_{j=1}^{N} \lambda_{j} (t) e_{j}(x), ~~(x,t)\in T^{d} \times [0,T]
\end{equation}
for some function $\lambda_{i}(t),$ and the norm of $u$ in $C^{0}([0,T];X_{N})$ can be formulated as
\begin{equation*}
\|u\|_{C^{0}([0,T];X_{N})} = \max_{t\in[0,T]} |\sum_{j=1}^{N} \lambda_{j}(t)|,
\end{equation*}
As a consequence, $u$ can be bounded in $C^{0}([0,T]; C^{k}(T^{d}))$ for any $k\in \mathbb{N}$, and there exists a constant $C>0$ depending on $k$ such that
\begin{equation}\label{2.2}
\|u\|_{C^{0}([0,T];C^{k}(T^{d}))} \leq C \|u\|_{C^{0}([0,T];L^{2}(T^{d}))}.
\end{equation}
Therefore there exists solution operator $F: C^{0}([0,T];X_{N})\rightarrow C^{0}([0,T];C^{3}(T^{d}))$ such that $n=F(u)$ be that the classical solution to
\begin{equation}\label{2.3}
n_{t} + {\rm div} (n u)= \varepsilon \Delta n,~n(x,0)=n_{0}~T^{d}\times(0,T)
\end{equation}
The maximum principle provides the lower and upper bounds
\begin{equation}\label{2.4}
\begin{aligned}
&\inf_{x\in T^{d}} n_{0}(x) \exp (-\int_{0}^{t}\|{\rm div} u\|_{L^{\infty}(T^{d})}ds) \leq n(x,t)\\
&\leq \sup_{x\in T^{d}} n_{0}(x) \exp (\int_{0}^{t}\|{\rm div} u\|_{L^{\infty}(T^{d})}ds),~for~(x,t)\in T^{d}\times [0,T].
\end{aligned}
\end{equation}
Since the equation is linear with respect to $n$, the operator $F$ is Lipschitz continuous in the following sense:
\begin{equation}\label{2.5}
\|F(v_{1})-F(v_{2})\|_{C^{0}([0,T];C^{k}(T^{d}))} \leq C \|v_{1}-v_{2}\|_{C^{0}([0,T];L^{2}(T^{d}))}.
\end{equation}
Since we assumed that $n_{0}(x)\geq \delta>0$, $n(t,x)$ is strictly positive. In view of (2.1), for $\|v\|_{C^{0}([0,T];L^{2}(T^{d}))}\leq c$, there exist constants $\underline{n}(c)$ and $\overline{n}(c)$ such that
\begin{equation}\label{2.6}
0 < \underline{n}(c, \varepsilon)\leq n(x,t) \leq \overline{n}(c, \varepsilon).
\end{equation}

Next, we wish to obtain the solvability of the magnetic field on the space $X_{N}$. To this end, for given $u$ above, we are looking for a unique function $B$ satisfying
\begin{subequations}\label{2.7}
\begin{eqnarray}
&\partial_{t} B - \nabla \times (u\times B) + \nabla \times (\nu_{b} \times B)=0,\label{1.1c}\\
& {\rm div} B=0,\label{1.1d}\\
& B(x,0)= B_{0}(x),\label{1.1e}
\end{eqnarray}
\end{subequations}
which is a linear parabolic-type problem in $B$. Therefore, by the standard Faedo-Galerkin methods, there exists a solution
\begin{equation}\label{2.8}
B\in L^{2}([0,T]; H^{1}(T^{3}))\cap L^{\infty}([0,T];L^{2}(T^{3})).
\end{equation}
to Eqs.(2.7). Further, there exists a continuous solution operator $G: C^{0}([0,T]; X_{N})\rightarrow L^{2}([0,T]; H^{1}(T^{3})) \cap L^{\infty}([0,T]; L^{2}(T^{3}))$ by $G(v)=B$.

Now, for all test function $\psi\in C([0,T];X_{N}) $ satisfying $\psi(\cdot,T)=0$, we wish to solve the momentum equation on the space $X_{N}$. To this end, for given $n=F(u), B=G(u)$, we are looking for a function $u_{N}\in C^{0}([0,T]; X_{N})$ such that
\begin{equation}\label{2.9}
\begin{aligned}
&-\int_{\Omega} n_{0}u_{0}\psi(\cdot,0)dx= \int_{0}^{T}\int_{T^{d}} ( nu_{N} \cdot \psi_{t} + (nu\otimes u_{N}): \nabla \psi+ P(n) {\rm div} \psi dx dt \\
&- \lambda \int_{0}^{T}\int_{T^{d}} \Delta^{s+1} (nu) : \Delta^{s}(n\psi) dx dt -\frac{\hbar^{2}}{2} \int_{0}^{T}\int_{T^{d}} (\varphi^{\prime}(n) \Delta \varphi(n) \nabla n \psi+ n \varphi^{\prime}(n) \Delta \varphi(n) {\rm div} \psi) dx dt\\
&-\lambda \int_{0}^{T} \int_{T^{d}} \Delta^{s} ({\rm div} (n \psi)): \Delta ^{s+1} n dx dt- 2\int_{0}^{T} \int_{T^{d}} \mu(n) D(u_{N}) \cdot \nabla \psi dx dt- \varepsilon \int_{0}^{T} \int_{T^{d}} (\nabla n \nabla \cdot) u \psi dx dt\\
& - \int_{0}^{T} \int_{T^{d}} \lambda(n) {\rm div} u_{N} \cdot {\rm div} \psi dx dt+ \mu_{b} \int_{0}^{T}\int_{T^{d}} (\nabla \times B)\times B \cdot \psi dxdt
\end{aligned}
\end{equation}
we will apply Banach's fixed point theorem to prove the local-in-time existence of solutions in the above equation. The regularization yields the $H^{1}$ regularity of $u_{N}$ which is needed to conclude the global existence of solutions.

To solve (2.9), we follow Ref.6 and consider a family of linear operators, given a function $\rho \in L^{1}(T^{d})$ with $\rho\geq \underline{\rho}>0$
\begin{equation*}
M[\rho]: X_{N}\rightarrow X_{N}^{\star},~~<M[n]v,u>=\int_{T^{d}} nv\cdot u dx, ~~v,u\in X_{N}.
\end{equation*}
where the symbol $X_{N}^{\star}$ stands for the dual space of $X_{N}$. It is easy to see that the operator $M$ is invertible provided $n$ is strictly positive on $T^{d}$ , and
\begin{equation*}
\| M^{-1}[n]\|_{L(X_{N}^{\star}, X_{N})} \leq \underline{\rho}^{-1},
\end{equation*}
where $L(X_{N}^{\star}, X_{N})$ is the set of bounded linear mappings from $X_{N}^{\star}$ to $X_{N}$. Moreover, the identity
\begin{equation*}
M^{-1}[n_{1}]-M^{-1}[n_{2}]= M^{-1}[n_{2}]
(M[n_{1}]-M[n_{2}])M^{-1}[n_{1}],
\end{equation*}
can be used to get
\begin{equation*}
\|M^{-1}[n_{1}]-M^{-1}[n_{2}]\|_{L(X_{N}^{\star}, X_{N})} \leq C(N,\underline{n}) \|n_{1}-n_{2}\|_{L^{2}(T^{d})},
\end{equation*}
for any $n_{1},n_{2}$ such that
\begin{equation*}
\inf_{T^{d}} n_{1} \geq n_{0} >0,~~\inf_{T^{d}} n_{2} \geq n_{0} >0,
\end{equation*}
So, $M^{-1}$ is Lipschitz continuous in the sense of (2.8).

Now the integral equation (2.9) can be rephrased as an ordinary differential equation on the finite-dimensional space $X_{N}$
\begin{equation*}
\frac{d}{dt}(M[n(t)u_{N}(t)])= N[v,u_{N},n,B]
\end{equation*}
where $n=F(u), B=G(u)$ and
\begin{equation}\label{2.10}
\begin{aligned}
&<N[u,u_{N},n,B]>= \int_{0}^{T} \int_{T^{d}} (n_{N}u_{N}\otimes u_{N}: \nabla \psi + (P(n_{N}+ P_{c}(n_{N})) {\rm div} \psi \\
& - \frac{\hbar^{2}}{2} (\varphi^{\prime}(n_{N})\Delta \psi(n_{N}) \nabla n_{N} \psi + n_{N} \varphi^{\prime} \Delta \varphi(n_{N}) {\rm div} \psi)dx+ \lambda \int_{0}^{T} \int_{T^{d}}n_{N} \nabla \Delta^{2s+1}(n_{N}u_{N})\psi dx dt\\
& \lambda \int_{0}^{T} \int_{T^{d}}  n_{N} \Delta^{s} {\rm div} (n_{N}\psi)\cdot \Delta^{s+1} n_{N} dx - 2 \int_{0}^{T} \int_{T^{d}} \mu(n_{N}) D(u_{N}) \cdot \nabla \psi dx\\
& -\int_{0}^{T} \int_{T^{d}}  \lambda(n_{N}) {\rm div} n_{N} \cdot {\rm div} \psi dx dt + \nu_{b}\int_{0}^{T} \int_{T^{d}} (\nabla \times B_{N})\times B_{N} \cdot \psi dx, ~\psi\in X_{N},
\end{aligned}
\end{equation}
The operator $N[u,u_{N},n,B]$, defined for every $t\in [0,T]$ as an operator from $X_{N}$ to $X_{N}^{\star}$ is continuous in time. Then the existence of a unique solution to (2.9) can be obtained by using standard theory for systems of ordinary equations. In other words, for given $u$, there exists a unique solution $u_{N}\in C^{1}([0,T];X_{N})$ to (2.7). Integrating (2.9) over (0,t) yields the following nonlinear equation:
\begin{equation}\label{2.11}
u_{N}= M^{-1}[F(u_{N})](t) (M[n_{0}] u_{0}+ \int_{0}^{t} N(u_{N},u_{N}(s),n_{N},B_{N}))dt
\end{equation}
in $X_{N}$. Because the operators $F, G, M^{-1}$ is Lipschitz continuous, this equation can be solved by evoking the fixed-pointed theorem of Banach on a short time interval $[0,T^{\prime}]$, where $T^{\prime}\leq T$, in the space $C^{0}([0,T^{\prime}];X_{N})$. In fact, we have even $u_{N}\in C^{0}([0,T^{\prime}];X_{N})$. Thus, there exists a unique local-in-time solution $(n_{N},u_{N},B_{N})$ to (2.2),(2.7) and (2.4).

\subsection{Global existence of solutions}

In order to prove that the solution $(n_{N},u_{N},B_{N})$ constructed above exists on the whole time interval $[0,T]$, it is sufficient to show that $u_{N}$ is bounded in $X_{N}$ on $[0,T^{\prime}]$ by employing the energy estimate.

\begin{lemm}
Let $T^{\prime}\leq T$, and let $n_{N}\in C^{1}([0,T^{\prime}];C^{3}(T^{d})), u_{N}\in C^{1}([0,T^{\prime}]; X_{N})$ and $B_{N}\in L^{2}([0,T^{\prime}]:H^{1}(T^{d}))\cap L^{\infty}([0,T^{\prime}]; L^{2}(T^{d}))$ be a local-in-time solution to (2.2),(2.7), and (2.4) with $n=n_{N}, u=u_{N},B=B_{N}$. Then
\begin{equation}\label{2.12}
\begin{aligned}
& \frac{d}{dt} E(n_{N}, u_{N}, B_{N}) + 2 \int_{T^{d}} \mu(n_{N})|\nabla u_{N}|^{2}dx + \int_{T^{d}} \lambda(n_{N})|{\rm div} u_{N}|^{2} dx + \varepsilon \int_{T^{d}} \frac{1}{n}(P^{\prime}(n)+P^{\prime}_{c}(n))|\nabla n|^{2} dx \\
&+ \nu_{b} \int_{T^{d}} |\nabla \times B_{N}|^{2} dx +  \lambda\int_{T^{d}} |\Delta^{s} \nabla (n_{N}u_{N})|^{2} dx+ \lambda \varepsilon \int_{T^{d}} |\Delta^{s+1} n_{N}|^{2}dx\\
&+\varepsilon\int_{T^{d}}\frac{\hbar^{2}}{2}\varphi^{\prime} (n_{N}) \Delta \varphi(n_{N}) \Delta n_{N} dx=0
\end{aligned}
\end{equation}
where
\begin{equation}\label{2.13}
\begin{aligned}
& E(n_{N},u_{N},B_{N})= \frac{1}{2}\int_{T^{d}} n_{N}|u_{N}|^{2}dx + \int_{T^{d}} [H(n_{N})+ H_{c}(n_{N})] dx\\
&+ \frac{\hbar^{2}}{2} \int_{T^{d}} |\nabla \varphi(u_{N})|^{2}dx +
\frac{1}{2}\int_{T^{d}} |B_{N}|^{2} dx + \frac{1}{2}\int_{T^{d}}\frac{\lambda}{2} |\nabla^{2s+1} n_{N}|^{2}dx,
\end{aligned}
\end{equation}
\end{lemm}

\begin{proof}
First we multipy (2.3) by $H^{\prime}(n_{N})- \frac{|u_{N}|^{2}}{2}- \frac{\hbar^{2}}{2}\varphi^{\prime} (n_{N}) \Delta \varphi(n_{N}) $, integrate over $T^{d}$, and integrate by parts:
\begin{equation}\label{2.14}
\begin{aligned}
&0= \int_{T^{d}} (\partial_{t}H(n_{N})- \frac{1}{2}|u_{N}|^{2}\partial_{t} n_{N} + \frac{\hbar^{2}}{2}\partial_{t}|\nabla \varphi(n_{N})|^{2}- n_{N}(H^{\prime\prime}(n_{N})+H^{\prime\prime}_{c}(n_{N}))\nabla n_{N}\cdot u_{N}\\
&+ n_{N}u_{N} \cdot \nabla u_{N} \cdot u_{N} - \frac{\hbar^{2}}{2}\varphi^{\prime} (n_{N}) \Delta \varphi(n_{N}) {\rm div} (n_{N}u_{N})+ \varepsilon H^{\prime\prime}(n_{N})|\nabla n_{N}|^{2} \\
& - \varepsilon \nabla n_{N} \cdot \nabla u_{N}\cdot u_{N}+ \varepsilon\frac{\hbar^{2}}{2}\varphi^{\prime} (n_{N}) \Delta \varphi(n_{N}) \Delta n_{N}) dx.
\end{aligned}
\end{equation}
Next, multipying the magnetic field equation (2.7) by $B_{N}$ we deduce that
\begin{equation}\label{2.15}
\int_{T^{d}} \nabla \times (u_{N}\times B_{N})\cdot B_{N} dx = \frac{1}{2} \int_{T^{d}}\frac{d}{dt} |B_{N}|^{2} dx+ \nu_{b} \int_{T^{d}} |\nabla \times B_{N}|^{2}dx,
\end{equation}
Then using the test function $u=u_{N}, n=n_{N}, B=B_{N}=G(u_{N})$ in (2.9) and integrating by parts leads to
\begin{equation}\label{2.16}
\begin{aligned}
&0= \int_{T^{d}} (|u_{N}|^{2}\partial_{t} n_{N} + \frac{1}{2} n_{N}\partial_{t}|u_{N}|^{2} - n_{N} u_{N} \otimes u_{N}: \nabla u_{N}+  (P^{\prime}(n_{N})+P^{\prime}_{c}(n_{N}) \nabla n_{N}\cdot u_{N}+ \frac{\lambda}{2}|\nabla^{2s+1} n|^{2} dx\\
&-2\int_{T^{d}}  {\rm div} (\mu(n_{N})D(u_{N}))u_{N} dx -\int_{T^{d}} \nabla (\lambda(n_{N}){\rm div} u_{N})\cdot u_{N} dx -\frac{\hbar^{2}}{2} \int_{T^{d}} n_{N} \nabla (\varphi^{\prime}\Delta\psi(n_{N})) u_{N} dx\\
&-
\nu_{b} \int_{T^{d}} (\nabla \times u_{N}) \times B_{N} \cdot B_{N} dx+\lambda\int_{T^{d}} |\Delta^{s} \nabla (n_{N}u_{N})|^{2} dx+ \lambda \varepsilon \int_{T^{d}} |\Delta^{s+1} n_{N}|^{2}dx
\end{aligned}
\end{equation}
Adding above three equations gives, since $n_{N}H^{\prime \prime}=p^{\prime}(n_{N})$.
Thus the proof of Lemma 2.1 is finished.
\end{proof}

From Lemma 2.1 we have the following estimates:

- the density estimates
\begin{equation}\label{2.17}
\begin{aligned}
&\|n_{N}\|_{L^{\infty}(0,T;L^{\gamma^{+}}(\Omega))} +  \|n_{N}\|_{L^{\infty}(0,T;L^{\gamma^{-}}(\Omega))}
+\|\nabla \varphi(n_{N})\|_{L^{\infty}(0,T;L^{2}(\Omega))} \\
&
+\sqrt{\varepsilon}\|\frac{1}{\sqrt{n_{N}}}\sqrt{\frac{\partial P_{c}}{\partial n_{N}} }\nabla n_{N}\|_{L^{2}(0,T;L^{2}(\Omega))} \leq C ,
\end{aligned}
\end{equation}

- the velocity estimates
\begin{equation}\label{2.18}
\begin{aligned}
&\|\sqrt{n_{N}}u_{N}\|_{L^{\infty}(0,T;L^{2}(\Omega))}+\|\sqrt{n_{N}}
D( u_{N})\|_{L^{2}(0,T;L^{2}(\Omega))}\\
&+
\|\sqrt{\lambda}\Delta^{s}\nabla(n_{N} u_{N})\|_{L^{2}(0,T;L^{2}(\Omega))}\leq C ,
\end{aligned}
\end{equation}

By a interpolation inequality we can get the density $\rho$ is bounded from below by a positive constant
\begin{equation}\label{2.19}
\begin{aligned}
&\|\rho^{-1}\|_{L^{\infty}((0,T)\times\Omega)}\leq \|\rho^{-1}\|_{L^{\infty}((0,T);H^{2}})
\\
&\leq C(1+\|\nabla^{3}\rho\|_{L^{\infty}((0,T);L^{2}(\Omega))})^{3}(1+\|\rho^{-1}\|_{L^{\infty}((0,T);L^{2}(\Omega))})^{4}
\leq C(\lambda) ,
\end{aligned}
\end{equation}
above we require $\gamma^{-}>4$ and $2s+1\geq 3$.

Combing with (2.15) we deduce the uniform bound for $\mathbf{u}$ , thus we get a global approximating solution.

The summarizing estimates (2.17)-(2.18) are  uniform with the dimension N, thus we can extract the weakly convergent subsequences and pass the limit passage $N\rightarrow\infty$ in the Galerkin approximation.

\section{\bf {Passage to the limit with N.}}
This subsection is devoted to the limit passage $N\rightarrow \infty $. Using estimates from the previous subsection we can extract weakly subsequences, whose limits satisfy the approximate system.
\subsection{\bf {Strong convergence of the density and passage to the limit in the continuity equation}}
From (2.17)-(2.18) we deduce that
\begin{equation}\label{3.1}
 u_{N}\rightarrow  u ~~~weakly ~~in ~ L^{2}(0,T;W^{2s+1,2}(\Omega))
\end{equation}
and
\begin{equation}\label{3.2}
n_{N}\rightarrow n ~~~weakly ~~in ~ L^{2}(0,T;W^{2s+2,2}(\Omega))
\end{equation}
at least for a suitable subsequence. In addition the r.h.s. of the linear parabolic problem
\begin{equation}\label{3.3}
\begin{aligned}
 \partial_{t} n+ div(n u)-\varepsilon\Delta n=0, \\
\rho(0,x)=\rho_{\lambda}^{0}(x),
\end{aligned}
\end{equation}
is uniformly bounded in $L^{2}(0,T;W^{2s,2}(\Omega))$ and the initial condition is sufficiently smooth, thus, applying he $L^{p}-L^{q}$ theory to this problem we conclude that $\{\partial_{t}\rho_{N}\}_{n=1}^{\infty}$ is uniformly bounded in  $L^{2}(0,T;W^{2s,2}(\Omega))$. Hence, the standard compact embedding implies $\rho_{N}\rightarrow \rho$ a.e. in $(0,T)\times\Omega$ and therefore passage to the limit in the approximate continuity equation is straightforward.

\subsection{\bf {Passage to the limit in the momentum equation}}
Having the strong convergence of the density, we start to identify the limit for $N\rightarrow\infty$ in the nonlinear terms of the momentum equation.

{\bf{The convective term}}.  First, one observes that
 \begin{equation*}
\rho_{N}\mathbf u_{N}\rightarrow \rho\mathbf u ~~~weakly^{*} ~~~in ~~~L^{\infty}(0,T;L^{2}(\Omega))
\end{equation*}
due to the uniform estimates (2.18) and the strong convergence of the density. Next, one can show that for any $\phi\in\cap_{n=1}^{\infty}X_{N}$ the family of functions $\int_{\Omega}\rho_{N}\mathbf u_{N}\phi dx$ is bounded and equi-continuous in $C(0,T)$, thus via the Arzela-Ascoli theorem and density of smooth functions in $L^{2}(\Omega)$ we get that
 \begin{equation}\label{3.4}
\rho_{N}\mathbf u_{N}\rightarrow \rho\mathbf u  ~~~in ~~~C([0,T];L^{2}_{weak}(\Omega))
\end{equation}
Finally, by the compact embedding $L^{2}(\Omega)\subset W^{-1,2}(\Omega)$ and the weak convergence of $\mathbf u_{N}$ we verify that
\begin{equation}\label{3.5}
\rho_{N}\mathbf u_{N}\otimes\mathbf u_{N}\rightarrow \rho\mathbf u \otimes\mathbf u ~~~~~~weakly~~~in ~~~L^{2}((0,T)\times\Omega)
\end{equation}
{\bf{The capillarity term}}. We write it in the form
\begin{equation*}
\int_{0}^{T}\int_{\Omega}\rho_{N}\nabla\Delta^{2s+1}\rho_{N}\cdot\phi dx dt=
\int_{0}^{T}\int_{\Omega}\Delta^{s}div(\rho_{N}\phi)\Delta^{s+1}\rho_{N} dx dt
\end{equation*}
Due to (2.18) and the boundedness of the time derivative of $\rho_{N}$, we infer that
\begin{equation}\label{3.6}
\rho_{N}\rightarrow\rho ~~~~strongly~~ in~~ L^{2}(0,T;W^{2s+1,2}(\Omega))
\end{equation}
thus
\begin{equation*}
\int_{0}^{T}\int_{\Omega}\Delta^{s}div(\rho_{N}\phi)\Delta^{s+1}\rho_{N}dx dt
\rightarrow \int_{0}^{T}\int_{\Omega}\Delta^{s}div(\rho\phi)\Delta^{s+1}\rho dx dt
\end{equation*}
for any $\phi\in C^{\infty}((0,T)\times\overline{\Omega})$.

{\bf{The momentum term}}. we write it in the form
\begin{equation*}
-\lambda\int_{0}^{T}\int_{\Omega}\rho_{N}\Delta^{2s+1}(\rho_{N}\mathbf u_{N})\cdot \phi dx dt=
-\lambda\int_{0}^{T}\int_{\Omega}\Delta^{s}\nabla(\rho_{N}\mathbf u_{N}):\Delta^{s}\nabla(\rho_{N}\mathbf \phi) dx dt
\end{equation*}
so the convergence established in (3.1) and (3.6) are sufficient to pass to the limit here.

Strong convergence of the density  enables us to perform in the momentum equation (2.9) satisfied for any function $\phi\in C^{1}([0,T];(X_{N}))$ such that $\phi(T)=0$ and by the density argument we can take all such test functions from $C^{1}([0,T];W^{2s+1}(\Omega))$.

\section{Derivation of the B-D estimate}
At this level we are left with only two parameters of approximation: $\varepsilon$ and $\lambda$.
From the so-far obtained a-priori estimates only the ones following from (2.17) and (2.18)
were independent of these parameters. Now we will have get more enough estimates for density and velocity from the B-D entropy energy inequality, we will prove the following lemma.
\begin{lemm} (Bresch-Desjardins type estimate)
The following identity holds:
\begin{equation}\label{4.1}
\begin{aligned}
& \frac{d}{dt} \int_{T^{d}} (\frac{1}{2} n|u+ \nabla \phi(n)|^{2}+ H(n)+H_{c}(n)+ \frac{\hbar^{2}}{2}|\nabla \varphi(n)|^{2}+ \frac{1}{2}|B|^{2}) dx +  \int_{T^{d}} 2\mu(n)|A(u)|^{2} dx \\
&+ \frac{\hbar^{2}}{2}\int_{T^{d}} \varphi^{\prime}(n)|\Delta \varphi(n)|^{2} dx+ \nu_{b} \int_{T^{d}} |\nabla \times B|^{2} dx + 2 \int_{T^{d}} \mu^{\prime}(n) (P^{\prime}(n)+P^{\prime}_{c}(n))\frac{|\nabla n|^{2}}{n} dx\\
&+2\lambda\int_{\Omega}  \Delta^{s+1} n \Delta^{s} \mu(n) dx =
-2\lambda\int_{\Omega} \Delta^{s}\nabla(n u): \Delta^{s} \nabla^{2}\mu(n) dx - \varepsilon\int_{\Omega}{\rm div}(nu)\phi^{'}(n)\Delta n dx\\
&+
\varepsilon\int_{\Omega}\frac{|\nabla \phi(n)|^2}{2}\Delta n dx
-\varepsilon\int_{\Omega}(\nabla n \cdot \nabla) u \cdot \nabla \phi(n) dx+\varepsilon\int_{\Omega} n\nabla \phi(n)\cdot \nabla (\phi^{'}(n)\Delta n) dx\\
&+ \int_{T^{d}} (\nabla\times B)\times B \cdot \nabla \phi(n) dx,
\end{aligned}
\end{equation}
in $\mathcal{D}^{\prime}(0,T)$, where $\nabla \phi(n)=2\frac{\nabla\mu(n)}{n} $.
\end{lemm}

\begin{proof} The basic idea of the proof is to find the explicit form of the terms:
\begin{equation}\label{4.2}
\frac{d}{dt}\int_{\Omega}(\frac{1}{2}n |u|^2+nu\cdot\nabla \phi(n)+ \frac{1}{2}n|\nabla \phi(n)|^2)dx.
\end{equation}
The first term can be evaluated by means of the main energy inequality, i.e.
\begin{equation}\label{4.3}\begin{aligned}
& \frac{d}{dt}\int_{T^{d}}(\frac{1}{2} n|u|^{2}+ [H(n)+ H_{c}(n)] dx+ \frac{\hbar^{2}}{2} |\nabla \varphi(u)|^{2} +
\frac{1}{2}|B|^{2} dx + \frac{\lambda}{2} |\nabla^{2s+1} n|^{2})dx\\
& +2 \int_{T^{d}} \mu(n)|\nabla u|^{2}dx + \int_{T^{d}} \lambda(n)|{\rm div} u|^{2} dx + \varepsilon \int_{T^{d}} \frac{1}{n}(P^{\prime}(n)+P^{\prime}_{c}(n))|\nabla n|^{2} dx \\
&+ \nu_{b} \int_{T^{d}} |\nabla \times B|^{2} dx +  \lambda\int_{T^{d}} |\Delta^{s} \nabla (nu)|^{2} dx+ \lambda \varepsilon \int_{T^{d}} |\Delta^{s+1} n|^{2}dx\\
&+\varepsilon\int_{T^{d}}\frac{\hbar^{2}}{2}\varphi^{\prime} (n) \Delta \varphi(n) \Delta ndx=0
,\end{aligned}
\end{equation}
To get a relevant expression for third term in (4.2), we multiply the approximate continuity equation by $\frac{|\nabla \phi(n)|^2}{2}$ and we obtain the following sequence of equalities
\begin{equation}\label{4.4}\begin{aligned}
& \frac{d}{dt}\int_{\Omega}\frac{1}{2}n|\nabla \phi(n)|^2dx \\
& =\int_{\Omega}(n \partial_{t}\frac{|\nabla \phi(n)|^2}{2}-\frac{|\nabla \phi(n)|^2}{2}div(n u)+\varepsilon\frac{|\nabla \phi(n)|^2}{2}\Delta n) dx \\
& = \int_{\Omega}(\rho\nabla \phi(n)\cdot \nabla (\phi^{'}(n)\partial_{t}n)-\frac{|\nabla \phi(n)|^2}{2}div(n  u)+\varepsilon\frac{|\nabla \phi(n)|^2}{2}\Delta n)dx
,\end{aligned}
\end{equation}
Using the approximate continuity equation, we get
\begin{equation}\label{4.5}\begin{aligned}
&\int_{\Omega}n\nabla \phi(n)\cdot \nabla (\phi^{'}(n)\partial_{t}n) dx \\
& =\int_{\Omega}\varepsilon n\nabla \phi(n)\cdot \nabla (\phi^{'}(n)\Delta n) dx-
\int_{\Omega} \rho\nabla \mathbf u : \nabla\phi(\rho)\otimes \nabla\phi(n) dx \\
&-\int_{\Omega}n\nabla \phi(n)\cdot \nabla (\phi^{'}(n)n div \mathbf u) dx
-\int_{\Omega} n\mathbf u \otimes \nabla\phi(n) :\nabla^{2}\phi(n) dx
,\end{aligned}
\end{equation}
Integrating by parts the two last terms from the r.h.s.
\begin{equation}\label{4.6}\begin{aligned}
&\int_{\Omega}n\nabla \phi(n)\cdot \nabla (\phi^{'}(n)\partial_{t}n) dx \\
& =\int_{\Omega}\varepsilon n\nabla \phi(n)\cdot \nabla (\phi^{'}(n)\Delta n) dx-
\int_{\Omega} n\nabla  u : \nabla\phi(n)\otimes \nabla\phi(n) dx \\
&+\int_{\Omega}n|\nabla \phi(n)|^{2} {\rm div} u dx +
\int_{\Omega} n^{2} \phi^{'}(n)\Delta \phi(n){\rm div} u dx \\
&+\int_{\Omega} |\nabla \phi(n)|^{2} {\rm div}(n  u) dx
+\int_{\Omega} n u \cdot\nabla( \nabla\phi(n)) \cdot\nabla(\phi(n)) dx
,\end{aligned}
\end{equation}
Combining the three previous equalities we finally obtain
\begin{equation}\label{4.7}\begin{aligned}
& \frac{d}{dt}\int_{\Omega}\frac{1}{2}n|\nabla \phi(n)|^2dx \\
& = \int_{\Omega}\varepsilon n\nabla \phi(n)\cdot \nabla (\phi^{'}(n)\Delta n) dx-
\int_{\Omega} n\nabla u : \nabla\phi(n)\otimes \nabla\phi(n) dx \\
&+\int_{\Omega}n|\nabla \phi(n)|^{2} {\rm div}u dx +
\int_{\Omega} n^{2} \phi^{'}(n)\Delta \phi(n){\rm div} u dx+\int_{\Omega}\varepsilon\frac{|\nabla \phi(n)|^2}{2}\Delta n dx
,\end{aligned}
\end{equation}
In the above series of equalities, each one holds ponitwisely with respect to time
due to the regularity of $n$ and $\nabla \phi$. This is not the case of the middle
 integrant of (4.2), for which one should really think of weak in time formulation. Denote
\begin{equation}\label{4.8}
V=W^{2s+1,2}(\Omega), ~~~~~~and ~~~~~\mathbf v=n \mathbf u,~~~~h=\nabla \phi.
\end{equation}
We know that $\mathbf v\in L^{2}(o,T;V)$ and its weak derivative with respect to time variable $\mathbf v^{'}\in L^{2}(o,T;V^{*})$ where
$V^{*}$ denotes the dual space to V. Moreover, $h\in L^{2}(0,T;V)$, $h^{'}\in L^{2}(0,T;W^{2s-1,2}(\Omega))$. Now, let $\mathbf v_{m}$, $h_{m}$ denote the standard mollifications in time of
$\mathbf v$, $h$ respectively. By the properties of mollifiers we know that
\begin{equation}\label{4.9}
\mathbf v_{m},\mathbf v_{m}^{'}\in C^{\infty}(0,T;V),~~~~~~~~~~~h_{m},h_{m}^{'}\in C^{\infty}(0,T;V)
\end{equation}
and
\begin{equation}\label{4.10}\begin{aligned}
\mathbf v_{m}\rightarrow v ~~~~~L^{2}(0,T;V),~~~~~h_{m}\rightarrow h~~~~~~L^{2}(0,T;V), \\
\mathbf v_{m}^{'}\rightarrow v^{'}~~~~~L^{2}(0,T;V^{*}),~~~~~h_{m}^{'}\rightarrow h^{'}~~~~~~L^{2}(0,T;V^{*})
,\end{aligned}
\end{equation}
For these regularized sequences we may write
\begin{equation}\label{4.11}\begin{aligned}
\frac{d}{dt}\int_{\Omega}\mathbf v_{m}\cdot h_{m} dx=\frac{d}{dt}(\mathbf v_{m},h_{m})_{V}=(\mathbf v_{m}^{'},h_{m})_{V}+(\mathbf v_{m},h_{m}^{'})_{V}, ~~~\forall \psi \in\mathcal{D}(0,T)
,\end{aligned}
\end{equation}
Using the Riesz representation theorem we verify that
$\mathbf v_{m}^{'}\in C^{\infty}(0,T;V)$ uniquely determines the functional
$\Phi_{v_{m}^{'}}\in V^{*}$ such that
$(\mathbf v_{m}^{'},\psi)_{V}=(\Phi_{v_{m}^{'}},\psi)_{V^{*},V}=\int_{\Omega}\mathbf v_{m}^{'}\cdot\psi dx$, $\forall \psi \in V$;
for the second term from the r.h.s. of (4.11) we can simply replace $V=L^{2}(\Omega)$ and thus we obtain
\begin{equation}\label{4.12}\begin{aligned}
-\int_{0}^{T}(\mathbf v_{m},h_{m})_{V}\psi^{'} dt= \int_{0}^{T}(\mathbf v_{m}^{'},h_{m})_{V^{*},V}\psi dt +\int_{0}^{T}(\mathbf v_{m},h_{m}^{'})_{L^{2}(\Omega)}\psi dt~~~~
\forall \psi \in \mathcal{D}(0,T)
,\end{aligned}
\end{equation}
Observe that both integrands from the r.h.s. are uniformly bounded in $L^{1}(0,T)$,
 thus, using (4.10), we let $m\rightarrow\infty$ to obtain
\begin{equation}\label{4.13}\begin{aligned}
\frac{d}{dt}\int_{\Omega}\mathbf v\cdot h dx=(\mathbf v^{'},h)_{V}+(\mathbf v,h^{'})_{V}, ~~~\forall \psi \in\mathcal{D}(0,T)
,\end{aligned}
\end{equation}
Coming back to our original notation, this means that the operation
\begin{equation}\label{4.14}
\frac{d}{dt}\int_{\Omega}nu\cdot\nabla \phi(n)dx =<\partial_{t}(n u),\nabla \phi>_{V^{*},V}+\int_{\Omega}n u\cdot \partial_{t}\nabla \phi dx
\end{equation}
is well defined and is nothing but equality between two scalar distributions. By the fact that $\partial_{t}\nabla \phi$ exists a.e. in $(0,T)\times \Omega$ we may use approximation to write
\begin{equation}\label{4.15}
\int_{\Omega}nu\cdot\partial_{t}\nabla \phi(n)dx =\int_{\Omega}({\rm div}(n u))^{2}\phi^{'}(n) dx- \varepsilon\int_{\Omega}{\rm div}(n u)\phi^{'}(n)\Delta n dx
\end{equation}
whence the first term on the r.h.s. of (4.14) may be evaluated by testing the approximate momentum equation by $\nabla n$
\begin{equation}\label{4.16}\begin{aligned}
&<\partial_{t}(n \mathbf u),\nabla \phi(n)>_{V^{*},V}= -\int_{\Omega}(2\mu(n)+\lambda(n))\Delta \phi(n) {\rm div} u dx +2\int_{\Omega}\nabla u: \nabla \phi(n) \otimes \nabla \mu(n) dx \\
&-2 \int_{\Omega}\nabla \phi(n)\cdot\nabla \mu(n) {\rm div} u dx -\int_{\Omega}\nabla \phi(n)\cdot \nabla P dx -\lambda\int_{\Omega} \Delta^{s+1}\mu(n) \Delta^{s} {\rm div} (n \nabla \phi(n))dx
\\
&-\lambda\int_{\Omega} \Delta^{s}\nabla(n  u): \Delta^{s} \nabla(n \nabla \phi(n))dx -\int_{\Omega}\nabla \phi(n)\cdot {\rm div}(n  u\otimes u) dx \\
&-\varepsilon\int_{\Omega}(\nabla n \cdot \nabla) u \cdot \nabla \phi(n) dx
-\int_{\Omega}\varphi^{\prime}(n) |\Delta \varphi(n)|^{2} dx+ \int_{T^{d}} (\nabla\times B)\times B \cdot \phi(n) dx
,\end{aligned}
\end{equation}
Recalling the form of $\phi(n)$ it can be deduced that the
\begin{equation}\label{4.17}\begin{aligned}
&\frac{d}{dt}\int_{\Omega}(n u\cdot\nabla \phi(n)+ \frac{1}{2}n|\nabla \phi(n)|^2)dx+ \int_{\Omega}\nabla \phi(n)\cdot \nabla P dx+\lambda\int_{\Omega} \mu^{\prime}(n) \Delta \mu(n) \Delta^{s} \mu(n) dx \\
& = -\int_{\Omega}\nabla \phi(n)\cdot {\rm div}(n u\otimes u) dx + \int_{\Omega}({\rm div}(nu))^{2}\phi^{'}(n) dx
-2\lambda\int_{\Omega} \Delta^{s}\nabla(n u): \Delta^{s} \nabla^{2}\mu(n) dx \\
&- \varepsilon\int_{\Omega}{\rm div}(nu)\phi^{'}(n)\Delta n dx+
\varepsilon\int_{\Omega}\frac{|\nabla \phi(n)|^2}{2}\Delta n dx
-\varepsilon\int_{\Omega}(\nabla n \cdot \nabla) u \cdot \nabla \phi(n) dx\\
&+\varepsilon\int_{\Omega} n\nabla \phi(n)\cdot \nabla (\phi^{'}(n)\Delta n) dx-\int_{\Omega}\varphi^{\prime}(n) |\Delta \varphi(n)|^{2} dx+ \int_{T^{d}} (\nabla\times B)\times B \cdot \nabla \phi(n) dx
,\end{aligned}
\end{equation}
The first two terms from the r.h.s. of (4.17) can be transformed
\begin{equation}\label{4.8}\begin{aligned}
&\int_{\Omega}(({\rm div}(n u))^{2}\phi^{'}(n) -\nabla \phi(n)\cdot {\rm div}(n u\otimes u)) dx \\
& = \int_{\Omega}(n^{2}\phi^{'}(n)({\rm div} u)^{2}+n\phi^{'}\cdot \nabla n {\rm div} u-n\phi^{'} \nabla n( u\cdot \nabla u))dx \\
& = 2\int_{\Omega}\mu(n) \partial_{i} u_{j}\partial_{j} u_{i}dx = 2\int_{\Omega}\mu(n) |D(u)|^{2} dx -  2\int_{\Omega}\mu(n) (\frac{\partial_{i} u_{j}-\partial_{j} u_{i}}{2})^{2}dx
,\end{aligned}
\end{equation}
thus, the assertion of Lemma 4.1 follows by adding (4.3) to (4.17).
\end{proof}

The main problem is to control the last term on the right hand side of (4.1), other terms can be easier to control. For this obstacle, we estimate as follow
\begin{equation}\label{4.19}
\begin{aligned}
& 2|\int_{T^{d}} (\nabla \times B)\times B \cdot \frac{\nabla \mu(n)}{n}|
\leq \int_{T^{d}}\frac{|\nabla \times B|^{2}}{\varepsilon n^{2}}dx+ \varepsilon \int_{T^{d}}|\nabla \mu(n)\times B|^{2} dx ,
\end{aligned}
\end{equation}
The first term of the right hand side will sent to the left hand side of equation and will we compensated with the term related to the resistivity thanks to the profiles condition introduced in (1.6).

The dimension hypothesis appears at this point, in a 2-dimensional space, we can insure that $W^{1,1}\subset L^{2}$ and this will be the main tool deal with the second term. we have
\begin{equation}\label{4.20}
\begin{aligned}
\|\nabla \mu(n)\times B  \|_{L^{2}(T^{d})}^{2} &\leq C\|\nabla \mu(n)\times B\|^{2}_{W^{1,1}}\\
& \leq C  (\|\Delta \mu(n)\|_{L^{2}(T^{d})}^{2}\|B\|_{L^{T^{d}}}^{2}+ \| \nabla \mu(n)\|_{L^{2}(T^{d})}^{2}\|\nabla B\|_{L^{2}(T^{d})}^{2}\\
&+ \|\nabla \mu(n)\times B\|_{L^{1}(T^{d})}^{2} ,
\end{aligned}
\end{equation}
But, from (2.12), we already know that $\|B\|_{L^{2}}$ and $\|\nabla \mu(n)\|_{L^{2}}$ are uniformly bounded by $\Lambda_{0}$, that is why we also get
\begin{equation}\label{4.21}
\|\nabla \mu(n)\times B  \|_{L^{2}(T^{d})}^{2} \leq C(1+\|\Delta \mu(n)\|_{L^{2}}^{2}+\|\nabla \times B\|^{2}_{L^{2}}),
\end{equation}
So we get, summing (4.20) and (4.21)  and taking into account all these quantities, for $\varepsilon$ small enough,
we are  considering here some coefficients $\varepsilon<\frac{1}{6}$ and such that $\mu^{\prime}-C\varepsilon$ still higher that a constant, say $\delta$. It also appears the necessary conditions on the constants $d_{0}$ and $d_{1}$, to be high enough because we need to have $\eta(n)-\varepsilon^{-1}n^{-2}-C\varepsilon \geq 0$.   To conclude, we apply a Gronwall's lemma, we will get B-D entropy energy estimates.

\section{Estimates independent of $\varepsilon,\lambda$, passage to the limit $\varepsilon,\lambda\rightarrow 0$}
In this section we first present the new uniform bounds arising fron the estimate of B-D entropy, performed in this section, and then we let the last two approximation parameters to 0. Note that the limit passage $\lambda\rightarrow 0$, $\varepsilon\rightarrow0$ could be done in a single step.

We complete the set uniform bounds by the following ones
\begin{equation}\label{5.1}
 \sqrt{\lambda} \|\Delta^{s+1} n_{\varepsilon,\lambda} \|_{L^{2}((0,T)\times\Omega)}+\| \nabla \phi(n_{\varepsilon,\lambda}) \|_{L^{2}((0,T)\times\Omega)}+
 \|\sqrt{\frac{\mu^{\prime}(n_{\varepsilon,\lambda})(P^{\prime}(n_{\varepsilon,\lambda})
 +P^{\prime}_{c}(n_{\varepsilon,\lambda}))}{n_{\varepsilon,\lambda}}}\nabla n\|_{L^{2}((0,T)\times\Omega)} \leq C
\end{equation}
moreover
\begin{equation}\label{5.2}
\|\Delta\mu(n_{\varepsilon,\lambda})\|_{L^{2}((0,T)\times\Omega)}\leq C
\end{equation}
The uniform estimates for the velocity vector field are the following ones
\begin{equation}\label{5.3}
 \sqrt{\lambda} \|\Delta^{s} \nabla(\rho u_{\varepsilon,\lambda}) \|_{L^{2}((0,T)\times\Omega)}+\|\sqrt{\mu(n_{\varepsilon,\lambda})}\nabla A(u_{\varepsilon,\lambda})\|_{L^{2}((0,T)\times\Omega)}\leq C
\end{equation}
and the constants from the r.h.s are independent of $\varepsilon$ and $\lambda$.

We now present several additional estimates of $n_{\varepsilon,\lambda}$ and $u_{\varepsilon,\lambda}$ based on imbedding of Sobolev spaces and simple interpolation inequalities.

\subsection{Further estimates of $n$}
\begin{lemm}
\begin{equation}\label{5.4}
n_{\varepsilon,\lambda}^{-1/2}~ is~ uniformly~ bounded~ in~ L^{\infty}(0,T;L^{6}_{loc}(\Omega))\cap L^{2}(0,T;H^{1}_{loc}(\Omega)),
\end{equation}
\begin{equation}\label{5.5}
n_{\varepsilon,\lambda}~ is~ uniformly~ bounded~ in~ L^{\infty}(0,T;L^{p}_{loc}(\Omega)),\forall p<+\infty.
\end{equation}
\end{lemm}

\begin{proof}
On the one hand, from (2.10) we know that $H_{c}(n_{\varepsilon,\lambda})$ is uniformly bounded in $L^{\infty}(0,T:L^{1}(\Omega))$ which implies that $n_{\varepsilon,\lambda}^{-1/2}$ is bounded in $L^{\infty}(0,T;L^{2\gamma^{-}})$. On the other hand, there exist functions $\zeta(n)=n$ for $n <1$,$\zeta(n)=0$ for $n>1$ such that $\nabla \zeta(n)^{-\frac{\frac{1}{}}{2}}$ is bounded in $L^{2}(0,T;L^{2}(\Omega))$. Then, thinking that $\gamma^{-}>1>\alpha$, we conclude that $\nabla n_{\varepsilon,\lambda}^{-1/2}$ is also bounded in $L^{2}(0,T;L^{2}(\Omega))$.

Since $\nabla \psi(n_{\varepsilon,\lambda})$ is bounded in $L^{\infty}(0,T;L^{2}(\Omega))$, and $H(n_{\varepsilon,\lambda})$ is uniformly bounded in $L^{\infty}(0,T;L^{1}(\Omega))$, thus we can use Sobolev embedding of $H^{1}(\Omega)$ in $L^{p}(\Omega)$ for all $p<+\infty$ in the two dimension.
\end{proof}

\subsection{Estimate of the velocity vector field}
\begin{lemm}
\begin{equation}\label{5.6}
u_{\varepsilon,\lambda}~ is~ uniformly~ bounded~ in~ L^{q_{1}}(0,T;W^{1,q_{2}}_{loc}(\Omega)),~~q_{1}>\frac{5}{3}~~ and~~ q_{2}>\frac{15}{8},
\end{equation}
\end{lemm}
\begin{proof}
 We use the Holder inequality to write
 \begin{equation}\label{5.7}
\|\nabla u_{\varepsilon,\lambda}\|_{L^{q_{1}}(0,T;L^{q_{3}}(\Omega))}\leq c(1+\|\zeta(n_{\varepsilon,\lambda})^{-\alpha/2}\|_{L^{2j}(0,T;L^{6j}(\Omega))})\|n_{\varepsilon,\lambda}^{\frac{\alpha}{2}} \nabla u_{\varepsilon,\lambda} \|_{L^{2}((0,T)\times \Omega)}.
\end{equation}
where $j=\frac{\gamma^{-}+1-\alpha}{\alpha}, \frac{1}{q_{1}}=\frac{1}{2}+\frac{1}{2j}, \frac{1}{q_{3}}= \frac{1}{2}+\frac{1}{6j}$. Therefore, the Korn inequality together with the Sobolev imbedding imply
the lemma.
\end{proof}
\subsection{Magnetic field}
Thanks to estimates (2.10) and conditions on $\eta$ that
\begin{equation}\label{5.8}
B_{\varepsilon,\lambda}~ is~ uniformly~ bounded~ in~ L^{\infty}(0,T;L^{2}(\Omega))\cap L^{2}(0,T;H^{1}(\Omega)),
\end{equation}
By interpolation, we can also deduce that embedding the following result:
\begin{lemm}
Let $\beta$ be any parameter in $(0,1)$ and $p<+\infty$.
\begin{equation}\label{5.8}
B_{\varepsilon,\lambda}~ is~ uniformly~ bounded~ in~ L^{\frac{2}{\beta}}(0,T;L^{\frac{2}{(\frac{2}{p})\alpha+1}}(\Omega)),
\end{equation}
\end{lemm}

\subsection{\bf {Passage to the limit with $\varepsilon\rightarrow 0$ and $\lambda\rightarrow 0$}}
With the B-D estimate at hand, especially with the bound on $\Delta^{s+1} n_{\varepsilon,\lambda}$ in $L^{2}((0,T)\times \Omega)$, which is now uniform with respect to $\varepsilon$, we may perform the limit passage similarly as in previous step. Indeed, the uniform estimates allow us to extract subsequences, such that
\begin{equation}\label{5.9}
\varepsilon \Delta^{s}\nabla u_{\varepsilon,\lambda}, \varepsilon \nabla n_{\varepsilon,\lambda}, \varepsilon \Delta^{s+1} n_{\varepsilon,\lambda}\rightarrow 0 ~~~ strongly ~~~~in~~~ L^{2}((0,T)\times \Omega)
\end{equation}
therefore
\begin{equation}\label{5.10}
\varepsilon  \nabla n_{\varepsilon,\lambda} \nabla u_{\varepsilon,\lambda} \rightarrow 0 ~~~ strongly ~~~~in~~~ L^{1}((0,T)\times \Omega)
\end{equation}

\subsection{For $n_{\varepsilon,\lambda}$}
We know, thinks to (5.5), that $n_{\varepsilon,\lambda}$ converges weakly to $n$ in $L^{\infty}(0,T; L^{q}_{loc}(\Omega))$, for all $q<+\infty$. To prove strong convergence on the density, we shall use the transport equation satisfying $\mu(n)$:
\begin{equation*}
\partial_{t}(\mu(n))+ {\rm div}(\mu(n) u)+ \frac{1}{2} \lambda(n){\rm div} u=0  ,
\end{equation*}
Proving that $\partial_{t}(\phi\mu(n))$ is bounded in $L^{2}(0,T; H^{-\sigma_{0}}(\Omega))$ for any compactly supported $\phi$, we then conclude that
\begin{equation}
n_{\varepsilon,\lambda}\rightarrow n ~~in~C([0,T]; L^{q}_{loc}(\Omega)),~~\forall q<+\infty,
\end{equation}
From another point, to conclude to a compactness for $n_{\varepsilon,\lambda}^{-1/2}$ in $C([0,T]; L^{q}_{loc}(\Omega))$, for all $q<+\infty$, we must, in addition to (5.4), look at $\partial_{t}(n^{-1/2})$ and try to show a boundedness in a space $L^{r}(0,T;H^{-\sigma_{0}})$ with $r>1$. From the transport equation we find
\begin{equation*}
\partial_{t}(n^{-1/2})-\frac{3}{2}n^{-1/2}{\rm div}u+{\rm div}(n^{-1/2}u)=0,
\end{equation*}
from which we can insure that $\partial_{t}(n^{-1/2})$ is bounded in $L^{5/3}(0,T;W^{-1,\frac{30}{11}}(\Omega))$. Then, from (5.4), we can deduce that
\begin{equation}
n_{\varepsilon,\lambda}^{-1/2}\rightarrow n^{-1/2} ~~in~L^{p}(0,T; L^{q}_{loc}(\Omega)),~~\forall p<+\infty,\forall q<6, ~~~in ~L^{2}(0,T; L^{q}_{loc}(\Omega)),~~\forall q<+\infty,
\end{equation}

\subsection{For $n_{\varepsilon,\lambda}u_{\varepsilon,\lambda}$}
We know that $n_{\varepsilon,\lambda}u_{\varepsilon,\lambda}$ converges weakly to $nu$ in $L^{\infty}(0,T; L^{s<2}_{loc}(\Omega))$ as the product of $n_{\varepsilon,\lambda}$ bounded in $L^{\infty}(0,T; L^{r<\infty}_{loc}(\Omega))$ and $\sqrt{n_{\varepsilon,\lambda}}u_{\varepsilon,\lambda}$ bounded in $L^{\infty}(0,T; L^{2}(\Omega))$. To have compactness on $n_{\varepsilon,\lambda}u_{\varepsilon,\lambda}$, we will of course use the momentum equation to assure that $\partial_{t}(n_{\varepsilon,\lambda}u_{\varepsilon,\lambda})$ is bounded in $L^{p}_{loc}(0,T; H^{-\sigma_{0}}(\Omega))$ for $p>1$ and $\sigma_{0}$ large enough. To more precise on what is different in our system we shall forget the new term in the momentum equation related to the magnetic field, namely $\nabla B\times B$. Using (5.8), we know that $\nabla\times B$ is bounded in $L^{2}(0,T;L^{2}(\Omega))$
, that is why we must have better than $L^{2}(0,T;L^{2}(\Omega))$ for $B$ and it is time to use Lemma 5.3. Indeed, for any $0<\alpha<1$ we get the expected boundedness of $B$ in $L^{p}(0,T;L^{p}(\Omega))$ with $p>2$ so that $(\nabla \times B)\times B$ is bounded in $B$ in $L^{q}(0,T;L^{q}(\Omega))$ with $q>1$. Thus, we get
\begin{equation}
n_{\varepsilon,\lambda}u_{\varepsilon,\lambda}\rightarrow nu ~~in~L^{p}(0,T; W^{-1,q}_{loc}(\Omega)),~~\forall p<+\infty,\forall q<6,
\end{equation}
From (5.14) together with lemma 5.2, is the strong convergence of $\int_{B}n_{\varepsilon,\lambda} |u_{\varepsilon,\lambda}|^{2}$ to $\int_{B}n |u|^{2}$ , for all subset $B$ in $\Omega$. Moreover, since $\sqrt{n_{\varepsilon,\lambda}}u_{\varepsilon,\lambda}$ converges weakly to $\sqrt{n}u$ in $L^{\infty}(0,T; L^{2}_{loc}(\Omega))$, we insure that
\begin{equation}
\sqrt{n_{\varepsilon,\lambda}}u_{\varepsilon,\lambda}\rightarrow \sqrt{n}u ~~in~L^{2}(0,T; L^{2}_{loc}(\Omega)),
\end{equation}
\subsection{For the magnetic field $B_{\varepsilon,\lambda}$}

We already know that the sequence $B_{\varepsilon,\lambda}$ weakly converges the limit $B$ in $L^{\infty}(0,T;L^{2}(\Omega))\cap L^{2}(0,T;H^{1}(\Omega))$. Let's now deal with $\partial_{t}B$ in order to insure a strong convergence statement. Looking at equation (1.1c), we are lead to bound $u\times B$ and $(\xi_{b})\nabla \times B$. For the first one, thinking to Lemma 5.2 and Lemma 5.3, we get $u\times B$ bounded in $L^{p}_{loc}(0,T;L^{p})$ with $p>1$ what is enough comfortable.
For the second, we write $\xi_{b}\nabla \times B=\sqrt{\xi_{b}}\sqrt{\xi_{b}}\nabla \times B$. We know that the term $\sqrt{\xi_{b}}\nabla \times B$ is bounded in $L^{2}(0,T;L^{2}(\Omega))$ and through conditions (1.6) and the bounds (37) or (38), we also $\sqrt{\xi_{b}}$ bounded in $L^{2}(0,T;L^{2}(\Omega))$. This is just enough to conclude that $B_{t}$ is bounded in $L^{1}(0,T;W^{-1,1}(\Omega))$. Then  we get
\begin{equation}
B_{\varepsilon,\lambda}\rightarrow B ~~in~L^{p}(0,T; L^{2}(\Omega)),~~\forall p<+\infty,
\end{equation}

\section{Convergences}

~~~~For the mass conservation, by the strong convergences of $n_{\varepsilon,\lambda}$ to n in $C([0,T];L^{2}(\Omega))$ and the strong convergence of $\sqrt{n_{\varepsilon,\lambda}}u_{\varepsilon,\lambda}$ in $L^{2}(0,T; L^{2}_{loc}(\Omega))$.

For the momentum equation, we have to justify how to pass the limit in the term $\nabla \times B_{\varepsilon,\lambda} \times B_{\varepsilon,\lambda}$. For that we should have a strong convergence of $B_{\varepsilon,\lambda}$to $B$ in $L^{2}(0,T; L^{2}_{loc}(\Omega))$.

Now is the time to deal with the magnetic field equation. It is clear for the term $\partial_{t}B$, now deal with $\nabla \times B \times B$ and $\nabla \times B \times B$.

With Lemma 5.2 and (5.16), we justify the convergence in the sense of distribution for the first one. The second one can be, one more time, written as the product of $\sqrt{\xi_{b}}\nabla\times B$ , weakly converging in $L^{2}(0,T; L^{2}_{loc}(\Omega))$ and $\sqrt{\eta}$ strongly convergence to $\sqrt{\eta}$ in $L^{2}(0,T; L^{2}_{loc}(\Omega))$.

\phantomsection

\end{document}